\documentclass[11pt]{amsart}
\usepackage{amsmath,amssymb,verbatim}
\usepackage[colorlinks,pagebackref,pdftex, bookmarks=false]{hyperref}

\numberwithin{equation}{section}

\newtheorem{thm}[equation]{Theorem}
\newtheorem{lemma}[equation]{Lemma}

\theoremstyle{definition}

\newcommand{\F}{\mathbb{F}}
\newcommand{\bP}{\mathbb{P}}
\newcommand{\Z}{\mathbb{Z}}

\begin{document}

\title[Permutation polynomials from R\'edei functions]
{Constructing permutation polynomials using generalized R\'edei functions}

\author{Zhiguo Ding}
\address{
  Hunan Institute of Traffic Engineering,
  Hengyang, Hunan 421001 China
}
\email{ding8191@qq.com}

\author{Michael E. Zieve}
\address{
  Department of Mathematics,
  University of Michigan,
  530 Church Street,
  Ann Arbor, MI 48109-1043 USA
}
\email{zieve@umich.edu}
\urladdr{http://www.math.lsa.umich.edu/$\sim$zieve/}


\date{\today}

\begin{abstract}
We determine all permutations in two large classes of polynomials over finite fields, where the construction of the polynomials in each class involves the denominators of a class of rational functions generalizing the classical R\'edei functions.  Our results generalize eight recent results from the literature, and our proofs of our more general results are much shorter and simpler than the previous proofs of special cases.
\end{abstract}

\maketitle


\section{Introduction}

In this paper we determine all permutation polynomials in two large families of polynomials.  Our main results are as follows, where $\mu_{q+1}$ denotes the set of $(q+1)$-th roots of unity in $\F_{q^2}^*$.

\begin{thm}\label{main1}
Pick positive integers $r$ and $n$ with $r\equiv n \pmod {q+1}$, distinct $u,v\in\mu_{q+1}$, and $a,b\in \F_{q^2}^*$, and write
\[
B(X) := a (X+u)^n + b (X+v)^n.
\]
Then  $f(X):=X^r B(X^{q-1})$ permutes $\F_{q^2}$  if and only if $(b/a)^{q-1}\ne (v/u)^n$ and
$\gcd(rn,q-1)=1$.
\end{thm}

\begin{thm}\label{main2}
Pick positive integers $r,n$ with $r\equiv n \pmod {q+1}$, and elements $a,b,v \in \F_{q^2}^*$ with $v\notin\mu_{q+1}$, and write
\[
B(X) := a (X+v^{-q})^n + b (X+v)^n.
\]
Then  $f(X):=X^r B(X^{q-1})$ permutes  $\F_{q^2}$ if and only if $bv^n/a \notin \mu_{q+1}$ and
$\gcd(r,q-1)=1=\gcd(n,q+1)$.
\end{thm}

We also show that special cases of Theorems~\ref{main1} and \ref{main2} imply \cite[Thms.~3.3 an 3.5]{WYDM}, \cite[Thms.~1.1 and 1.2]{ZR}, and corrected versions of \cite[Thms.~1 and 2]{FFLW} and \cite[Thms.~8 and 9]{PUW}.

The reason several authors have studied instances of the polynomials in Theorems~\ref{main1} and \ref{main2}, besides the simple description of these polynomials, is that in some sense these polynomials are produced from the important class of rational functions known as R\'edei functions.  These rational functions were originally introduced in \cite{R} over fields of odd order, and in \cite{N} over fields of even order.  R\'edei functions have the form $\rho(X)\circ X^n\circ\eta(X)$ for certain degree-one rational functions $\rho,\eta\in\F_{q^2}(X)$.  The polynomials $B(X)$ in Theorems~\ref{main1} and \ref{main2} are the denominators of rational functions of this form in which $\rho(X)$ and $\eta(X)$ are allowed to be in a more general class of degree-one rational functions than those used in the definition of R\'edei functions.  Since R\'edei functions induce permutations of certain finite fields, it is perhaps intuitively plausible that polynomials constructed from denominators of R\'edei functions might also induce permutations, but this connection is not immediate, and becomes even more mysterious when instead of R\'edei functions we use generalized R\'edei functions.

The connection between R\'edei functions and permutation polynomials was first made in \cite{ZR}, and then further developed in \cite{FFLW, PUW, WYDM}.  All of these papers use certain instances of the classical R\'edei functions.  In this paper we make a threefold generalization, by treating all classical R\'edei functions rather than just specific instances, by addressing both odd and even $q$, and by using the generalized R\'edei functions discussed above.  We obtain results that are much more general than those in previous papers.  
Our proofs combine some parts of the strategy in \cite{ZR} with further ideas.  This yields especially simple arguments which are significantly shorter than the previous proofs of special cases of our results.  We note that  a quick
conceptual treatment of R\'edei functions is in \cite{DZtrig}, and a
comprehensive account will appear in our forthcoming paper \cite{DZRedei}.

This paper is organized as follows.  In the next section we present our notation and recall some easy known results.  Then we prove Theorems~\ref{main1} and \ref{main2} in Section~\ref{secproofs}, and we conclude the paper in Section~\ref{secreferences} by describing the specific special cases of Theorems~\ref{main1} and \ref{main2} which imply eight previous results in the literature.


\section{Background results}

In this section we recall some easy known results. 


\subsection{Notation}

We use the following notation and terminology 
throughout this paper:
\begin{itemize}
\item $q$ is a fixed prime power;
\item $\mu_{q+1}$ denotes the set of $(q+1)$-th roots of unity in $\F_{q^2}$;
\item $\bP^1(\F_q) := \F_q\cup\{\infty\}$;
\item for any $g(X)\in\F_{q^2}(X)$, we write $g^{(q)}(X)$ for the rational function obtained from $g(X)$ by replacing 
each coefficient by its $q$-th power;
\item the \emph{degree} of a nonzero rational function $g(X)$ is the maximum of the degrees of $N(X)$ and $D(X)$, 
for any prescribed choice of coprime polynomials $N(X)$ and $D(X)$ such that $g(X)=N(X)/D(X)$.
\end{itemize}


\subsection{Permutations and roots of unity}

The following result is a special case of \cite[Lemma~2.1]{Zlem}.

\begin{lemma}\label{old}
Write $f(X)=X^r B(X^{q-1})$ where $r$ is a positive integer, $q$ is a prime power, and $B(X)\in\F_{q^2}[X]$. Then $f(X)$ 
permutes\/ $\F_{q^2}$ if and only if $\gcd(r,q-1)=1$ and $g_0(X):=X^r B(X)^{q-1}$ permutes $\mu_{q+1}$.
\end{lemma}

The following result encodes a procedure introduced in \cite{ZR}, which is spelled out in \cite[Lemma~2.2]{Zx}.

\begin{lemma}\label{lemx}
Let $q$ be a prime power, and write $g_0(X)=X^r B(X)^{q-1}$ where $r\in\Z$ and $B(X)\in\F_{q^2}[X]$. Then $g_0(X)$ permutes 
$\mu_{q+1}$ if and only if $B(X)$ has no roots in $\mu_{q+1}$ and $g(X):=X^r B^{(q)}(1/X)/B(X)$ permutes $\mu_{q+1}$.
\end{lemma}

\begin{lemma}\label{scr}
Write $g(X):=X^n B^{(q)}(1/X)/B(X)$ for some nonzero $B(X)\in\F_{q^2}[X]$ and some integer $n$ with $n\ge\deg(B)$.  If $\deg(g)=n$ then $B(X)$ has no roots in $\mu_{q+1}$ and $g(\mu_{q+1})\subseteq\mu_{q+1}$.
\end{lemma}

\begin{proof}
Note that $g(X)=N(X)/B(X)$ where $N(X):=X^n B^{(q)}(1/X)$, and that the hypothesis $n\ge\deg(B)$ implies that $N(X)\in\F_{q^2}[X]$.
Since 
\[
\max(\deg(B),\deg(N))\le n=\deg(g)\le\max(\deg(B),\deg(N)),
\]
the second inequality must be an equality, so that $B(X)$ and $N(X)$ are coprime.
It follows that $B(X)$ has no roots in $\mu_{q+1}$, since if $\alpha\in\mu_{q+1}$ is a root of $B(X)$ then
$0=B(\alpha)^q=B^{(q)}(1/\alpha)$, so that $\alpha$ is also a root of $N(X)$, contradicting coprimality of $B(X)$ and $N(X)$. Thus for any $\alpha\in\mu_{q+1}$ we have
\[
g(\alpha)^q= \Bigl(\frac{ \alpha^n B^{(q)}(1/\alpha)}{B(\alpha)} \Bigr)^q = \frac{B(\alpha)} {\alpha^n B^{(q)}(1/\alpha)} = \frac1{g(\alpha)},
\]
so that $g(\alpha)\in\mu_{q+1}$ and thus $g(\mu_{q+1})\subseteq\mu_{q+1}$.
\end{proof}


\subsection{Degree-one rational functions}

We begin with the well-known and easily proved description of degree-one rational functions.

\begin{lemma}\label{who1}
For $\alpha,\beta,\gamma,\delta\in\F_{q^2}$ with $\{\gamma,\delta\}\ne\{0\}$, the rational function $(\alpha X+\beta)/(\gamma X+\delta)$ has degree $1$ if and only if $\alpha\delta\ne \beta\gamma$.
\end{lemma}

Next we note that the degree of a rational function is unchanged by composing on both sides with degree-one rational functions:

\begin{lemma}\label{deg}
For any nonconstant $g(X)\in\F_{q^2}(X)$, and any degree-one $\rho,\eta\in\F_{q^2}(X)$, the degree of $\eta(g(\rho(X)))$ is $\deg(g)$.
\end{lemma}

\begin{proof}
This is easy to verify directly.  Alternately, it follows from the fact that if $x$ is transcendental over $\F_{q^2}$ and $h(X)\in\F_{q^2}(X)$ is nonconstant then $\deg(h)=[\F_{q^2}(x):\F_{q^2}(h(x))]$ (e.g.\ see \cite[Lemma~2.2]{DZAA}).
\end{proof}

The next two results are reformulations of \cite[Lemmas~2.1 and 3.1]{ZR}.

\begin{lemma}\label{deg1mu}
If $\alpha,\beta\in\F_{q^2}$ satisfy $\alpha^{q+1}\ne\beta^{q+1}$, then $(\beta^qX+\alpha^q)/(\alpha X+\beta)$ 
permutes $\mu_{q+1}$.
\end{lemma}

\begin{lemma}\label{mu}
If $\alpha\in\F_{q^2}\setminus\F_q$ and $\beta\in\mu_{q+1}$, then $(\alpha X+\beta\alpha^q)/(X+\beta)$ maps $\mu_{q+1}$ bijectively onto\/ $\bP^1(\F_q)$.
\end{lemma}


\section{Proofs of main results} \label{secproofs}

In this section we prove Theorems~\ref{main1} and \ref{main2}.

\begin{proof}[Proof of Theorem~\ref{main1}]
By Lemmas~\ref{old} and \ref{lemx}, $f(X)$ permutes $\F_{q^2}$ if and only if $\gcd(r,q-1)=1$, $B(X)$ has no roots in $\mu_{q+1}$, and
the rational function $g(X):=X^n B^{(q)}(1/X)/B(X)$ permutes $\mu_{q+1}$.  Plainly $B(X)\ne 0$ and
\[
X^n B^{(q)}(1/X) = a^qu^{-n} (X+u)^n + b^qv^{-n} (X+v)^n,
\]
so that
\[
g(X) = \frac{a^qu^{-n}X + b^qv^{-n}} {aX+b}\circ X^n \circ \frac{X+u}{X+v}.
\]
By Lemma~\ref{who1}, $g(X)$ is constant (and hence not bijective) when $a^q u^{-n}b = b^qv^{-n}a$, or equivalently $(b/a)^{q-1}=(v/u)^n$.  Now suppose $(b/a)^{q-1}\ne (b/u)^n$, so that $(a^q u^{-n} X + b^q v^{-n})/(aX+b)$ and $(X+u)/(X+v)$ have degree one by Lemma~\ref{who1}, and thus $\deg(g)=n$ by Lemma~\ref{deg}.
Then Lemma~\ref{scr} implies that $B(X)$ has no roots
in $\mu_{q+1}$, and that $g(\mu_{q+1})\subseteq\mu_{q+1}$.
Hence $g(X)$ permutes $\mu_{q+1}$ if and only if $g(X)$ is injective on $\mu_{q+1}$.
Pick $w$ with $w^{q-1}=u/v$.  Then $w \in \F_{q^2}\setminus\F_q$ and
\[
\frac{X+u}{X+v} = w^{-1}\frac{wX+vw^q}{X+v} = w^{-1}\rho(X),
\]
where $\rho(x):=(wX+vw^q)/(X+v)$  induces a bijection $\mu_{q+1}\to \bP^1(\F_q)$ by Lemma~\ref{mu}.  Hence  $g(X) = \eta(X)\circ X^n\circ\rho(X)$ for some degree-one $\eta(X)\in\F_{q^2}(X)$, so that $g(X)$ is
injective on $\mu_{q+1}$ if and only if $X^n$ is injective on $\bP^1(\F_q)$, or equivalently $\gcd(n,q-1)=1$.
\end{proof}

\begin{proof}[Proof of Theorem~\ref{main2}]
By Lemmas~\ref{old} and \ref{lemx}, $f(X)$ permutes $\F_{q^2}$ if and only if $\gcd(r,q-1)=1$, $B(X)$ has no roots in $\mu_{q+1}$, and
the rational function $g(X):=X^n B^{(q)}(1/X)/B(X)$ permutes $\mu_{q+1}$.  Plainly $B(X)\ne 0$ and
\[
X^n B^{(q)}(1/X) = a^qv^{-n} (X+v)^n + b^q v^{qn} (X+v^{-q})^n,
\]
so that
\[
g(X) = \frac{a^qv^{-n} X + b^q v^{qn}}{bX+a}\circ X^n\circ\frac{X+v}{X+v^{-q}}.
\]
By Lemma~\ref{who1}, $g(X)$ is constant (and hence not bijective) if $a^{q+1}v^{-n}=b^{q+1}v^{qn}$, or equivalently $bv^n/a\in\mu_{q+1}$.
Now suppose $bv^n/a\notin\mu_{q+1}$, so that $(a^qv^{-n} X + b^q v^{qn})/(bX+a)$ and $(X+v)/(X+v^{-q})$ have degree $1$ by Lemma~\ref{who1}, and thus $\deg(g)=n$ by Lemma~\ref{deg}.
Then Lemma~\ref{scr} implies that $B(X)$ has no roots
in $\mu_{q+1}$, and that $g(\mu_{q+1})\subseteq\mu_{q+1}$.
Hence $g(X)$ permutes $\mu_{q+1}$ if and only if $g(X)$ is injective on $\mu_{q+1}$.
Here  $\rho(X) := (X + v)/(v^q X+1)$ permutes $\mu_{q+1}$ by Lemma~\ref{deg1mu}, so that
$g(X)$ is injective on $\mu_{q+1}$ if and only if $X^n$ is injective on $\mu_{q+1}$, or equivalently $\gcd(n,q+1)=1$.
\end{proof}


\section{Connection with previous results} \label{secreferences}

The combination of \cite[Thms.~1 and 2]{FFLW} (when restricted to the case $n+m(q+1)>0$ in order to make them true)
are contained in the combination of the following special cases of our results:
\begin{itemize}
\item Theorem~\ref{main1} with both $b=\pm a$ and $v=-u$; 
\item Theorem~\ref{main2} with both $b=\pm a$ and $v^{q+1}=-1$.
\end{itemize}
One counterexample to \cite[Thms.~1 and 2]{FFLW} is $n=1=-m$, in which case the conditions on $n$ and $m$ in those results are satisfied but the function $P(x)$ in those results does not permute $\F_{q^2}$.

A corrected version of \cite[Thm.~8]{PUW} is contained in the special case of Theorem~\ref{main1} in which $q$ is even, $a=b=1$, and $\{u,v\}=\F_4\setminus\F_2$.  Our result in this special case has the permutation conditions $\gcd(rn,q-1)=1$ and $3\nmid n$.  By contrast, the conditions in \cite[Thm.~8]{PUW} are $\gcd(r,q-1)=1=\gcd(n,q^2-1)$.  One counterexample to \cite[Thm.~8]{PUW} is $q=32$, $n=11$, and $r=n+q+1$ (or in the notation of \cite{PUW}, $m=1$), which yields a permutation of $\F_{q^2}$ but violates the condition $\gcd(n,q^2-1)=1$ in \cite[Thm.~8]{PUW}.

A corrected version of \cite[Thm.~9]{PUW} is contained in the special case of Theorem~\ref{main1} in which $q$ is even, $a=v$, $b=u$, and $\{u,v\}=\F_4\setminus\F_2$.  Our result in this special case has the permutation conditions $\gcd(rn,q-1)=1$ and $n\not\equiv 2 \pmod 3$.  By contrast, the conditions in \cite[Thm.~9]{PUW} say instead (writing $r=n+m(q+1)$) ``$\gcd(n+m(q+1),q-1)=1$, and $\gcd(n,q^2-1)=1$, when $n\equiv 1 \pmod 3$".  Although this phrase is ambiguous, every plausible interpretation is contradicted by the example $q=128$, $n=43$, and $r=n+q+1$ (so $m=1$), which yields a permutation of $\F_{q^2}$ with $n\equiv 1\pmod 3$ but $\gcd(n,q^2-1)\ne 1$.  Here we also note that the identity displayed in the statement of \cite[Thm.~9]{PUW} is false.  It would be true if one defined the numbers $a_i$ in the same way as in \cite[Thm.~8]{PUW}.
In order to avoid confusion caused by this incorrect identity, we note that our corrected version of \cite[Thm.~8]{PUW} uses the definition of $N_n(x)$ at the bottom of \cite[p.~5]{PUW}, together with the definition of $N_n(x,\alpha)$ at the bottom of \cite[p.~4]{PUW}.

The special case of Theorem~\ref{main1} with $a=(-b)^n$ and $u=vb^{q-1}$ yields \cite[Thm.~3.5]{WYDM} and \cite[Thm.~1.2]{ZR}.
The special case of Theorem~\ref{main2} with $a=-1/v$ and $b=1/v^n$ yields \cite[Thm.~3.3]{WYDM} and \cite[Thm.~1.1]{ZR} after composing on both sides with scalar multiples.  

Finally, we note that \cite{FFLW} and \cite{PUW} emphasize that R\'edei functions can be computed recursively, via a method equivalent to computing $n$-th powers by $n-1$ multiplications.  A much faster method for computing R\'edei functions, analogous to the method of computing $n$-th powers by repeated squaring (which takes less than $2\log_2(n)$ multiplications), was given in \cite{More}.



\end{document}